\documentclass[11pt]{amsart}
\usepackage{amsmath,amsthm,amssymb,amsfonts}
\usepackage[all]{xy}
\usepackage[pdftex]{graphicx}

\newlength{\hchng}
\newlength{\vchng}
\newtheorem{theorem}{Theorem}
\newtheorem{proposition}{Proposition}
\newtheorem{lemma}{Lemma}
\newtheorem{conjecture}{Conjecture}
\newtheorem{corollary}{Corollary}

\newcommand{\Mgn}{M_{g,n}}
\newcommand{\barm}{\overline M}
\newcommand{\Mgnbar}{\barm_{g,n}}
\newcommand{\sm}{\mathcal M}
\newcommand{\sbarm}{\overline {\mathcal M}}
\newcommand{\aut}{\operatorname{Aut}}
\newcommand{\CC}{\mathbb{C}}
\newcommand{\PP}{\mathbb{P}}
\newcommand{\QQ}{\mathbb{Q}}
\newcommand{\ZZ}{\mathbb{Z}}
\newcommand{\td}{\widetilde{D}}

\begin{document}

\title{Divisors in the moduli spaces of curves}
\author{Enrico Arbarello}
\address{Dipartimento di Matematica ``Guido Castelnuovo'', Universit\`a di Roma ``La Sapienza'', P.le Aldo Moro 2, 00185 Roma, Italy}
\email{ea@mat.uniroma1.it}

\author{Maurizio Cornalba}
\address{Dipartimento di Matematica ``Felice Casorati'', Universit\`a di Pavia, Via Ferrata 1, 27100 Pavia, Italy}
\email{maurizio.cornalba@unipv.it}
\thanks{Research partially supported by PRIN 2007 \textit{Spazi di moduli e teoria di Lie}}

\date{}

\maketitle

\section{Introduction}\label{intro}

The calculation by Harer \cite{harer} of the second homology groups of the moduli spaces of smooth curves over $\CC$ can be regarded as a major step towards the understanding of the enumerative geometry of the moduli spaces of curves \cite{mumnum,konts}. However, from the point of view of an algebraic geometer, Harer's approach has the drawback of being entirely transcendental; in addition, his proof is anything but simple. It would be desirable to provide a proof of his result which is more elementary, and algebro-geometric in nature. While this cannot be done at the moment, as we shall explain in this note it is possible to reduce the transcendental part of the proof, at least for homology with rational coefficients, to a single result, also due to Harer \cite{harer_2}, asserting that the homology of $M_{g,n}$, the moduli space of smooth $n$-pointed genus $g$ curves, vanishes above a certain explicit degree. A sketch of the proof of Harer's vanishing theorem, which is not at all difficult, will be presented in Section \ref{high} of this survey. It must be observed that Harer's vanishing result is an immediate consequence of an attractive algebro-geometric conjecture of Looijenga (Conjecture \ref{aff} in Section \ref{high}); an affirmative answer to the conjecture would thus give a completely algebro-geometric proof of Harer's theorem on the second rational homology of moduli spaces of curves.

In this note we describe how one can calculate the first and second rational (co)homology groups of $M_{g,n}$, and those of $\barm_{g,n}$, the moduli space of stable $n$-pointed curves of genus $g$, using only relatively simple algebraic geometry and Harer's vanishing theorem. For $\barm_{g,n}$, this program was carried out in \cite{ac}, where the third and fifth cohomology groups were also calculated and shown to always vanish; in Section \ref{divclasses}, we give an outline of the argument, which uses in an essential way a simple Hodge-theoretic result due to Deligne \cite{deligne3}. In genus zero, we rely on Keel's calculation of the Chow ring of $\barm_{0,n}$; a simple proof of Keel's result in the case of divisors is presented in Section \ref{genus0}. We finally give a new proof of Harer's theorem for $H^2(M_{g,n};\QQ)$; we also recover Mumford's result asserting that $H^1(M_{g,n};\QQ)$ always vanishes for $g\ge 1$. The idea is to use Deligne's Gysin spectral sequence from \cite{deligne2}, applied to the pair consisting of $\barm_{g,n}$ and its boundary $\partial M_{g,n}$. This is possible since $\partial M_{g,n}$ is a divisor with normal crossings in $\barm_{g,n}$, if the latter is regarded as an orbifold. Roughly speaking, the Gysin spectral sequence calculates the cohomology of the open variety $M_{g,n}=\barm_{g,n}\smallsetminus \partial M_{g,n}$ in terms of the cohomology of the strata of the stratification of $\barm_{g,n}$ by ``multiple intersections'' of local components of $\partial M_{g,n}$. Knowing the first and second cohomology groups of the completed moduli spaces $\barm_{g,n}$ makes it possible to explicitly compute the low terms of the spectral sequence, and to conclude. Knowing the first and second homology of the moduli spaces of curves allows one to also calculate the Picard groups of the latter, as done for instance in \cite{acpic}.

\medskip
\section{Boundary strata in $\barm_{g,n}$}\label{strata}

As customary, we denote by $\sbarm_{g,n}$ the moduli stack of stable $n$-pointed genus $g$ curves, and by $\barm_{g,n}$ the corresponding coarse moduli space. It will be notationally convenient to allow the marked points to be indexed by an arbitrary set $P$ with $n=|P|$ elements, rather than by $\{1,\dots,n\}$. The corresponding stack and space will be denoted by $\sbarm_{g,P}$ and $\barm_{g,P}$. Of course, we shall write $\sm_{g,P}$ and $M_{g,P}$ to indicate the open substack and subspace parametrizing smooth curves. By abuse of language, we shall usually view $\sbarm_{g,P}$ and $\sm_{g,P}$ as complex orbifolds.

As is well known, to any stable $P$-pointed curve $C$ of genus $g$ one may attach a graph $\Gamma$, the so-called \it dual graph\rm, as follows. The vertices of $\Gamma$ are the components of the normalization $N$ of $C$, while the half-edges of $\Gamma$ are the points of $N$ mapping to a node or to a marked point of $C$. The edges of $\Gamma$ are the pairs consisting of half-edges mapping to the same node, while the half-edges coming from marked points are called legs. The vertices joined by an edge $\{\ell,\ell'\}$ are those which correspond to the components containing $\ell$ and $\ell'$. The dual graph comes with two additional decorations; the legs are labelled by $P$, and to each vertex $v$ there is attached a non-negative integer $g_v$, equal to the genus of the corresponding component of $N$. We shall denote by $V(\Gamma)$, $X(\Gamma)$, $E(\Gamma)$ the sets of vertices, half-edges, and edges of $\Gamma$, respectively. The following formula holds:
\begin{equation*}
g=h^1(\Gamma)+\sum_{v\in V(\Gamma)}g_v\,.
\end{equation*}
This implies, in particular, that $g$ depends only on the combinatorial structure of $\Gamma$; we are thus justified in calling it the genus of $\Gamma$. The stability condition for $C$ is $2g-2+|P|>0$, and hence can be stated purely in terms of $\Gamma$. We shall say that $\Gamma$ is a stable $P$-pointed graph of genus $g$. Given another $P$-pointed genus $g$ graph $\Gamma'$, an isomorphism between $\Gamma$ and $\Gamma'$ consists of bijections $V(\Gamma)\to V(\Gamma')$ and $X(\Gamma)\to X(\Gamma')$ respecting the graph structures, that is, carrying edges to edges, legs labelled by the same letter into each other, and vertices into vertices of equal genus.

Moduli spaces can be stratified by graph type. By this we mean the following. Fix a stable $P$-pointed genus $g$ graph $\Gamma$. For each vertex $v$ let $P_v$ be the subset of $P$ consisting of all elements labeling legs emanating from $v$, and denote by $H_v$ the set of the half-edges originating from $v$ which are not legs. In particular,
\begin{equation*}
P=
\bigcup_{v\in V(\Gamma)}P_v\,.
\end{equation*}
We denote by $D_\Gamma$ the closure of the locus of points in $\sbarm_{g,P}$ representing stable curves with dual graph isomorphic to $\Gamma$. It easy to see that $D_\Gamma$ is a reduced sub-orbifold of $\sbarm_{g,P}$ and that, in suitable local coordinates, it is locally a union of coordinate linear subspaces of codimension $|E(\Gamma)|$. We denote by $\td_\Gamma$ the normalization of $D_\Gamma$; by what we just observed, $\td_\Gamma$ is smooth.
We also set
\begin{equation*}
\sm_\Gamma=\prod_{v\in V(\Gamma)} \sm_{g_v,P_v\cup H_v}\,,\qquad
\sbarm_\Gamma=\prod_{v\in V(\Gamma)} \sbarm_{g_v,P_v\cup H_v}
\end{equation*}
We may define clutching morphisms
\begin{equation*}
\xi_\Gamma\colon \sbarm_\Gamma
\to \sbarm_{g,P}
\end{equation*}
as follows (cf. \cite{Knudsen}, page 181, Theorem 3.4).
Let $x$ be a point of $\sbarm_{g,P}$, consisting of a $P_v\cup H_v$-pointed curve $C_v$ for each vertex $v\in V(\Gamma)$. Then $\xi_\Gamma(x)$ corresponds to the curve obtained from the disjoint union of the $C_v$ by identifying the points labelled by $\ell$ and $\ell'$, for any edge $\{\ell,\ell'\}$ of $\Gamma$. By construction, the image of
$\xi_\Gamma$ is supported on $D_\Gamma\subset \sbarm_{g,P}$.
The automorphism group $\aut(\Gamma)$ acts on $\sbarm_\Gamma$ in the obvious way. Again by construction, $\xi_\Gamma$ induces a morphism
\begin{equation*}
\widetilde{\xi}_\Gamma\colon \sbarm_\Gamma/ \aut(\Gamma)
\to \sbarm_{g,P}\,,
\end{equation*}
which induces by restriction an isomorphism between $\sm_\Gamma/\aut(\Gamma)$ and an open dense substack of $D_\Gamma$. More generally, one can see that $\sbarm_\Gamma/ \aut(\Gamma)$ is isomorphic to the normalization of $D_\Gamma$, that is, to $\td_\Gamma$.

The graphs giving rise to codimension one strata, that is, the graphs with only one edge, are easily described. There is a single such graph $\Gamma_{irr}$ with only one vertex, and it is customary to denote the corresponding divisor with $D_{irr}$. The other graphs have two vertices, and are all of the form $\Gamma_{a,A}$, where $a$ is an integer such that $0\le a\le g$, and $A$ is a subset of $P$. A point of the corresponding divisors, usually denoted by $D_{a,A}$, consists of an $A$-pointed genus $a$ curve attached at a single point to a ($P\smallsetminus A$)-pointed curve of genus $g-a$.

The union of $D_{irr}$ and of the $D_{a,A}$ is just the boundary $\partial\sm_{g,P}$, that is, the substack of $\sbarm_{g,P}$ parametrizing singular stable curves. This is a \it normal crossings divisor \rm in the sense of stacks, which just means that, in suitable ``local coordinates'', it is locally a union of coordinate hyperplanes. More generally, for any integer $p$, the union of all strata $D_\Gamma$ such that $|E(\Gamma)|=p$ is the locus of points of multiplicity at least $p$ in $\partial\sm_{g,P}$.

Finally, suppose that a stable $P$-pointed genus $g$ graph $\Gamma'$ is obtained from another one, call it $\Gamma$, by contracting a certain number of edges.
In this case we write $\Gamma'<\Gamma$. There are natural morphisms
\begin{equation*}
\xi_{\Gamma',\Gamma}\colon \sbarm_\Gamma\to\sbarm_{\Gamma'}\,,
\end{equation*}
defined as follows. Suppose for simplicity that $\Gamma'$ is obtained from $\Gamma$
by contracting to a point $w$ a set $S$ of edges forming a \it connected \rm subgraph. We define a new graph $\Sigma$ as follows. The edges of $\Sigma$ are those in $S$, the vertices of $\Sigma$ are the end-vertices of edges in $S$, and the legs of $\Sigma$ are the legs of $\Gamma$ originating from vertices of $\Sigma$, or the half edges in $H_v$ that are halves of edges in $E(\Gamma)\smallsetminus S$; the set of the latter will be denoted $H'_v$.
We then have
\begin{equation*}
\aligned
\sbarm_\Sigma&=\prod_{v\in V(\Sigma)} \sbarm_{g_v,P_v\cup H'_v}\\
\sbarm_\Gamma&=\left(\prod_{v\in V(\Gamma)\smallsetminus V(\Sigma)} \sbarm_{g_v,P_v\cup H_v}\right)\times
\sbarm_\Sigma\\
\sbarm_{\Gamma'}&=\left(\prod_{v\in V(\Gamma)\smallsetminus V(\Sigma)} \sbarm_{g_v,P_v\cup H_v}\right)\times
\sbarm_{g_w,Q_w}
\endaligned
\end{equation*}
where
\begin{equation*}
g_w=h^1(\Sigma)+\sum_{v\in V(\Sigma)} g_v\,,\qquad Q_w=\bigcup_{v\in V(\Sigma)}\left(P_v\cup H'_v\right)
\end{equation*}
We then set
\begin{equation*}
\xi_{\Gamma',\Gamma}=(1, \xi_\Sigma)\,,\qquad \xi_\Sigma\colon \sbarm_\Sigma\to \sbarm_{g_w,P_w}\,.
\end{equation*}
By definition we have
\begin{equation*}
\xi_{\Gamma'}\circ\xi_{\Gamma',\Gamma}=\xi_{\Gamma},\qquad\text{if}\quad\Gamma'<\Gamma
\end{equation*}
An important case is the one in which $D_\Gamma$ is a codimension one stratum in $D_{\Gamma'}$. Suppose that there are exactly $k$ edges in
$\Gamma$ having the property that contracting one of them produces a graph isomorphic to $\Gamma'$, and let $F$ be the set of these edges. Then there are exactly $k$ branches of $D_{\Gamma'}$ passing through $D_{\Gamma}$.
In general, there is no natural morphism from $\td_\Gamma$ to $\td_{\Gamma'}$. The two are however connected as follows. Clearly, $F$ is stable under the action of the automorphism group of $\Gamma$. The group $\aut(\Gamma)$ then acts on $\sbarm_\Gamma\times F$ via the product action, and we set
\begin{equation*}
\td_{\Gamma',\Gamma}=(\sbarm_\Gamma\times F)/ \aut(\Gamma)
\end{equation*}
There are two natural mappings originating from $\td_{\Gamma',\Gamma}$: the projection $\pi_{\Gamma',\Gamma}\colon \td_{\Gamma',\Gamma}\to \td_\Gamma=\sbarm_\Gamma/ \aut(\Gamma)$, and the morphism
\begin{equation*}
\widetilde{\xi}_{\Gamma',\Gamma}\colon \td_{\Gamma',\Gamma}\to \td_{\Gamma'}
\end{equation*}
defined as follows. Given a point of $\sbarm_\Gamma$ and an edge $e\in F$, clutching along $e$ produces a point in $\sbarm_{\Gamma'}$, well defined up to automorphisms of $\Gamma'$. This defines a morphism $\sbarm_\Gamma\times F\to \td_{\Gamma'}$, and it is a simple matter to show that in fact this morphism factors through $\td_{\Gamma',\Gamma}$.

\medskip
\section{Tautological classes and relations}\label{tauto}

In the Chow ring and in the cohomology ring of the moduli spaces of curves there are certain natural, or \it tautological\rm, classes. Here we describe those of complex codimension one. First of all, we have the classes of the components of the boundary of $\sbarm_{g,P}$, that is, of the suborbifolds $D_{irr}$ and $D_{a,A}$ introduced in the previous section. We denote these classes by $\delta_{irr}$ and $\delta_{a,A}$, respectively. We write $\delta_b$ to indicate the sum of all the classes $\delta_{a,A}$ such that $a=b$, and $\delta$ to indicate the total class of the boundary, that is, the sum of $\delta_{irr}$ and of all the classes $\delta_{a,A}$. We also set $A^c=P\smallsetminus A$. Next, consider the projection morphism
\begin{equation}\label{defpi}
\pi_{x}\colon \sbarm_{g,P\cup\{x\}}\to \sbarm_{g,P}\,,
\end{equation}
and denote by $\omega=\omega_{\pi_x}$ the relative dualizing sheaf. For every $p\in P$, there is a section $\sigma_p\colon \sbarm_{g,P}\to \sbarm_{g,P\cup\{x\}}$, whose image is precisely $D_{0,\{p,x\}}$. The remaining tautological classes on $\sbarm_{g,P}$ that we will consider are
\begin{equation*}
\aligned
\psi_p&=\sigma_p^*(\omega)\,,\qquad p\in P\,,\\
\kappa_1&={\pi_x}_*(\psi_x^2)\,,
\endaligned
\end{equation*}
plus the Hodge class $\lambda$, which is just the first Chern class of the locally free sheaf ${\pi_x}_*(\omega)$. The Hodge class is related to the others by {\it Mumford's relation} (cf. \cite{mumens}, page 102, just before the statement of Lemma 5.14)
\begin{equation*}
\kappa_1=12\lambda-\delta+\psi\,,
\end{equation*}
where $\psi=\sum \psi_p$; we will not further deal with it in this section.

Consider the clutching morphisms $\xi_{\Gamma_{irr}}$ and $\xi_{\Gamma_{a,A}}$ described in the previous section. For convenience, we shall denote them by $\xi_{irr}$ and $\xi_{a,A}$, respectively. Thus
\begin{equation}\label{defxi}
\aligned
&\xi_{irr}\colon \sbarm_{g-1,P\cup\{q,r\}}\to\sbarm_{g,P}\,,\\
&\xi_{a,A}\colon \sbarm_{a,A\cup\{q\}}\times\sbarm_{g-a,(P\smallsetminus A)\cup\{r\}}\to\sbarm_{g,P}\,.
\endaligned
\end{equation}
We would like to describe the pullbacks of the tautological classes under the morphisms (\ref{defpi}) and (\ref{defxi}). It turns out that the formulas for the pullback under $\xi_{a,A}$ are somewhat messy. On the other hand, for our purposes it will suffice to give pullbacks formulas for the simpler map
\begin{equation}\label{deftheta}
\vartheta\colon\sbarm_{a,A\cup\{q\}}\to\sbarm_{g,P}
\end{equation}
which associates to any $A\cup\{q\}$-pointed genus $a$ curve the $P$-pointed genus $g$ curve obtained by glueing to it a {\it fixed} $A^c\cup\{r\}$-pointed genus $g-a$ curve $C$ via identification of $q$ and $r$. The following result is proved in \cite{ac}, Lemmas 3.1, 3.2, 3.3; for simplicity, in the statement we write $\xi$ in place of $\xi_{irr}$, and $\pi$ in place of $\pi_x$.

\begin{lemma}\label{pullbacks}The following pullback formulas hold:
\begin{itemize}
\item[i)]$\pi^*(\kappa_1)=\kappa_1-\psi_q$;
\item[ii)]$\pi^*(\psi_p)=\psi_p-\delta_{0,\{p,q\}}$ for any $p\in P$;
\item[iii)]$\pi^*(\delta_{irr})=\delta_{irr}$;
\item[iv)]$\pi^*(\delta_{a,A})=\delta_{a,A}+\delta_{a,A\cup\{q\}}$;

\item[v)]$\xi^*(\kappa_1)=\kappa_1$;
\item[vi)]$\xi^*(\psi_p)=\psi_p$ for any $p\in P$;
\item[vii)]$\xi^*(\delta_{irr})=\delta_{irr}-\psi_q-\psi_r
+\displaystyle{\sum_{q\in B, r\not\in B}}\delta_{b,B}$;
\item[viii)]$\xi^*(\delta_{a,A})=\left\{
\begin{matrix}\delta_{a,A}\hfill&&\text{if }
g=2a,\ A=P=\emptyset,\hfill\\
\delta_{a,A}+\delta_{a-1,A\cup\{q,r\}}\hfill&&\text{otherwise}.\hfill
\end{matrix}\right.$

\item[ix)]$\vartheta^*(\kappa_1)=\kappa_1$;
\item[x)]$\vartheta^*(\psi_p)=\left\{
\begin{matrix}\psi_p\hfill&&\text{if }p\in
A,\hfill\\
0\hfill&&\text{if }p\in A^c;\hfill
\end{matrix}\right.$
\item[xi)]$\vartheta^*(\delta_{irr})=\delta_{irr}$.
\end{itemize}

\noindent Suppose $A=P$. Then
\begin{itemize}
\item[xii)]$\vartheta^*(\delta_{b,B})=\left\{
\begin{matrix}
\delta_{2a-g,P\cup\{q\}}-\psi_q\hfill&&\text{if }(b,B)=(a,P)\text{ or }
(b,B)=(g-a,\emptyset),\hfill\\
\delta_{b,B}+\delta_{b+a-g,B\cup\{q\}}\hfill&&\text{otherwise}.\hfill
\end{matrix}
\right.$
\end{itemize}

\noindent Suppose $A\neq P$. Then
\begin{itemize}
\item[xii')]$\vartheta^*(\delta_{b,B})=\left\{
\begin{matrix}
-\psi_q\hfill&&\text{if }(b,B)=(a,A)\text{ or }
(b,B)=(g-a,A^c),\\
\delta_{b,B}\hfill&&\text{if }B\subset A\text{ and }(b,B)\neq (a,A),\hfill\\
\delta_{b+a-g,(B\smallsetminus A^c)\cup\{q\}}\hfill&&\text{if }B\supset A^c\text{ and
}(b,B)\neq (g-a,A^c),\hfill\\ 0\hfill&&\text{otherwise}.\hfill
\end{matrix}
\right.$
\end{itemize}
\end{lemma}
\noindent As shown in \cite{ac}, it follows from Lemma \ref{pullbacks} that in low genus there are linear relations between $\kappa_1$, the classes $\psi_p$, and the boundary classes. For instance, in genus zero
\begin{equation}\label{rel0}
\psi_z=\sum_{z\in A\not\ni x,y}\delta_{0,A}\,.
\end{equation}
This formula will be needed in the next section; for the remaining relations we refer to the statement of Theorem \ref{coh_bar}, where they appear as formulas (\ref{kappag2}), (\ref{kappag1}), (\ref{psig1}), and the first formula in (\ref{kpsidirr_rel}).

\medskip
\section{Divisor classes in $\barm_{0,n}$}\label{genus0}

In \cite{Keel}, Keel describes the Chow ring
of the moduli
space of pointed curves of genus 0. For brevity, when writing the boundary divisors of
$\barm_{0,P}$ we will drop the reference to the genus ($g=0$) and we will write
$D_S$ instead of $D_{0,S}$.
Similarly, for the divisor classes we will write
$\delta_S$ instead of $\delta_{0,S}$. Keel's theorem is the following.

\begin{theorem}[Keel \cite{Keel}]{\label{Keel}}The Chow
ring $A^*(\barm_{0,P})$ is generated
by the classes $\delta_S$, with $S\subset P$ and $|S| \geq 2$, $|S^c| \geq 2$. The relations among these generators
are generated by the following:
\begin{itemize}
\item [1)]$\delta_S=\delta_{S^c}$,
\item[2)] For any quadruple of distinct elements $i,j,k,l\in P$,
\begin{equation*}
\sum_{i,j\in S;\,k,l\notin S}\delta_S=\sum_{i,k\in S;\,j,l\notin S}\delta_S=\sum_{i,l\in S;\,j,k\notin S}\delta_S\,.
\end{equation*}
\item[3)] $\delta_S\delta_T=0$, unless $S\subset T$, $S\supset T$, $S\subset T^c$ or $S\supset T^c$.
\end{itemize}
Moreover, $A^*(\barm_{0,P})=H^*(\barm_{0,P}; \ZZ)$.
\end{theorem}

\noindent
We will not give a proof of this theorem. Instead, after a few general comments, we will give the complete computation of the first Chow group $A^1(\barm_{0,P})$,
showing that it coincides with $H^2(\barm_{0,P}; \ZZ)$.

\medskip
The relations 1), 2) and 3) can be easily proved. Relations 1) are obvious.
Relations 3) follow immediately from the fact that $D_S$ and $D_T$ do not physically meet except in the cases mentioned; alternatively, one can use part xii') of Lemma \ref{pullbacks}. To get the relations in 2) look at the morphism
\begin{equation*}
\pi_{i,j,k,l}\colon \barm_{0,P}\to \barm_{0,\{i,j,k,l\}}
\end{equation*}
defined by forgetting all the points in $P$ with the exception of $i$, $j$, $k$, $l$ and stabilizing the resulting curve.
Look at the divisor classes $\delta_{\{i,j\}}$ and $\delta_{\{i,k\}}$ on $\barm_{0,\{i,j,k,l\}}$.
The pull-back of $\delta_{\{i,j\}}$ and $\delta_{\{i,k\}}$, via $\pi_{i,j,k,l}$ are given, respectively, by
\begin{equation*}
\sum_{i,j\in S;\,k,l\notin S}\delta_S\,,\qquad\sum_{i,k\in S;\,j,l\notin S}\delta_S\,.
\end{equation*}
The fact that $\delta_{\{i,j\}}=\delta_{\{i,k\}}\in A^1(\barm_{0,\{i,j,k,l\}})=A^1(\PP^1)$ gives, by pull-back, the first
relation in 2). The second is obtained in a similar way.

We now concentrate our attention on the first Chow group $A^1(\barm_{0,P})$. Set $n=|P|$. Recall first that $M_{0,P}$ parametrizes ordered $n$-tuples of distinct points of $\PP^1$, modulo automorphisms. Fixing the first three points to be $0,1,\infty$ kills all automorphisms; hence $M_{0,P}$ can be identified with the space of ordered $(n-3)$-tuples of distinct points of $\PP^1\smallsetminus \{0,1,\infty\}$, that is, with the complement of the big diagonal in $(\CC\smallsetminus \{0,1\})^{n-3}$. In particular, $M_{0,P}$ is the complement of a union of hyperplanes in $\PP^{n-3}$. It follows that the Picard group of $M_{0,P}$ is trivial, and that $\barm_{0,P}$ is birationally equivalent to $\PP^{n-3}$. As a consequence, $h^{0,1}(\barm_{0,P})=h^{0,2}(\barm_{0,P})=0$, and hence
\begin{equation*}
A^1(\barm_{0,P})=\operatorname{Pic}(\barm_{0,P})=H^2(\barm_{0,P};\ZZ)\,.
\end{equation*}
We now show that $\operatorname{Pic}(\barm_{0,P})$ is generated by boundary divisors.
Let $L$ be a line bundle on $\barm_{0,P}$. Since the Picard group of $M_{0,P}$ is zero, the restriction of $L$ to $M_{0,P}$ is trivial, i.e., there is a meromorphic
section $s$ of $L$ which is regular and nowhere vanishing when restricted to $M_{0,P}$. Therefore the divisor of $s$ is of the form
$(s)=\sum n_iD_i$, where the $D_i$ are boundary divisors, so that $L=\mathcal O(-\sum n_iD_i)$. We conclude that
$H^2(\barm_{0,P};\ZZ)=A^1(\barm_{0,P})$ is generated
by the classes $\delta_S$, with $S\subset P$ and $|S| \geq 2$, $|S^c| \geq 2$, with the following relations
\begin{itemize}
\item[1)] $\delta_S=\delta_{S^c}$;
\item[2)] for any quadruple of distinct elements $i,j,k,l\in P$,
\begin{equation*}
\sum_{i,j\in S;\,k,l\notin S}\delta_S=\sum_{i,k\in S;\,j,l\notin S}\delta_S=\sum_{i,l\in S;\,j,k\notin S}\delta_S\,.
\end{equation*}
\end{itemize}
As a consequence, Keel's theorem for $H^2(\barm_{0,P};\ZZ)=A^1(\barm_{0,P})$ is implied by the following result.

\begin{proposition}{\label{basis_h_2}} Let $P$ be a finite set with $n=|P|\geq 4$ elements. Fix distinct elements $i,j,k\in P$.
Then $H^2(\barm_{0,P};\ZZ)$ is freely generated by the classes
\begin{equation}
\delta_S\,,\quad \text{with}\ \ i\in S\ \ \text{and}\ \ 2\le |S|\leq n-3\ \ \text{or}\ \ S=\{j,k\}\,.\label{basis_2}
\end{equation}
In particular
\begin{equation*}
\operatorname{rank} \left(H^2(\barm_{0,P};\ZZ)\right)=\operatorname{rank} \left(A^1(\barm_{0,P})\right)=2^{n-1}- {{n}\choose{2}}-1\,.
\end{equation*}
\end{proposition}
\begin{proof}Let $F$ be the subgroup of $H^2(\barm_{0,P};\ZZ)$ generated by the elements in (\ref{basis_2}). The only boundary divisors not appearing in (\ref{basis_2}) are the divisors $\delta_{\{s,t\}}$, where
$s\neq i$, $t\neq i$ and $\{s,t\}\neq \{j,k\}$. We write relation 2) for the quadruple $i,k,s,t$
\begin{equation*}
\delta_{\{i,k\}}+\delta_{\{s,t\}}+\sum_{i,k\in S;\,s,t\notin S,\,\,3\leq|S|\leq n-3}\delta_S=\delta_{\{i,t\}}+\delta_{\{s,k\}}+\sum_{i,t\in S;\,k,s\notin S,\,\,3\leq|S|\leq n-3}\delta_S\,.
\end{equation*}
Thus $\delta_{s,t}\equiv \delta_{s,k}$ modulo $F$.
Now starting from the quadruple $i,j,k,s$ we get $\delta_{s,k}\equiv 0$ modulo $F$.

Denote by $\mathcal F^P_{i,j,k}$ the set of all elements of the form (\ref{basis_2}).
We must prove that the elements of $\mathcal F^P_{i,j,k}$ are linearly independent. We proceed by induction on $n=|P|$.
The case $n=4$ is trivial. We assume $n\geq 5$.
Suppose there is a relation
\begin{equation}
\nu\delta_{j,k}+\sum_{i\in S,\,2\le|S|\leq n-3} \nu_S\delta_S=0\label{rel}
\end{equation}
Fix $a\in P\smallsetminus \{i\}$. We are going to pull back this relation to $\barm_{0, P\cup\{x\}\smallsetminus\{i,a\}}$
via the map $\vartheta$ defined in (\ref{deftheta}), where now $A=P\smallsetminus\{i,a\}$ and $g=0$. Let us first assume that $a\neq j$ and $a\neq k$.
Recalling parts xii) and xii') of Lemma \ref{pullbacks}, the pull-back of the LHS of (\ref{rel}) via $\vartheta$ is given by
\begin{equation}
\sum_{S\supset\{i,a\},\,|S|\leq n-3}\nu_S\delta_{S\cup\{x\}\smallsetminus\{i,a\}}-\nu_{\{i,a\}}\psi_x+\nu\delta_{\{j,k\}}\label{rel_pull}
\end{equation}
We now use the expression for $\psi_x$ given by (\ref{rel0}), where we consider $\psi_x$ as an element of
$H^2(\barm_{0, P\cup\{x\}\smallsetminus\{i,a\}})$. We get
\begin{equation*}
\psi_x=\delta_{\{j,k\}} +\sum_{x\in T\subset P\cup\{x\}\smallsetminus\{i,a\}\atop{j,k\in T^c},\,{2\le |T|\leq n-4}}\delta_T\,.\label{psi_expr}
\end{equation*}
The relation (\ref{rel_pull}) becomes
\begin{equation*}
\sum_{S\supset\{i,a\},\,|S|\leq n-3\atop j\in S\,\,\text{or}\,\,k\in S}
\nu_S\delta_{S\cup\{x\}\smallsetminus\{i,a\}}+\sum_{S\supset\{i,a\},\,|S|\leq n-3\atop j,k\in S^c}
(\nu_S-\nu_{\{i,a\}})\delta_{S\cup\{x\}\smallsetminus\{i,a\}}+(\nu-\nu_{\{i,a\}})\delta_{\{j,k\}}
\end{equation*}
We may now apply the induction hypothesis to $\mathcal F^{P\cup\{x\}\smallsetminus\{i,a\}}_{x,j,k}$.
Since $a$ is arbitrary, as long as $a\neq j$ and $a\neq k$ we deduce that
\begin{equation*}
\aligned
\nu_{\{i,a\}}=\,& \nu\,,\quad\text{for}\quad a\neq j,\, a\neq k\,;\\
\nu_S=\,& 0\,,\quad\text{for}\quad j\in S\,\,\text{or}\,\,k\in S,\,i\in S\,,
S\neq \{i,j\}\,, S\neq \{i,k\}\quad\text{and}\quad |S|\leq n-3\,; \\
\nu_S=\,& \nu\,, \quad\text{for}\quad j,k\in S^c ,\,i\in S\,, \quad\text{and}\quad |S|\leq n-3\,.
\endaligned
\end{equation*}
Using again the general expression for $\psi_i$ in terms of the boundary divisors given by (\ref{rel0}), the original relation (\ref{rel}) can be written as
\begin{equation*}
\nu_{\{i,j\}}\delta_{\{i,j\}}+\nu_{\{i,k\}}\delta_{\{i,k\}}+\nu\Big(\sum_{2\le |S|\leq n-3\atop i\in S;\, j,k\in S^c}
\delta_S \Big) +\nu\delta_{\{j,k\}}=
\nu_{\{i,j\}}\delta_{\{i,j\}}+\nu_{\{i,k\}}\delta_{\{i,k\}}+\nu\psi_i=0
\end{equation*}
We pull back this relation to $\barm_{P\cup\{x\}\smallsetminus\{j,k\}}$ and we get $\nu=0$.
Let $l\neq k$. We then pull back the resulting relation to $\barm_{\{i,j,l,x\}}$ and we get $\nu_{\{i,j\}}=0$.
But then $\nu_{\{i,k\}}=0$ as well.
\end{proof}

\medskip
\section{The vanishing of the high degree homology of $M_{g,n}$}\label{high}

We now begin our computation of the first and second cohomology groups of the moduli spaces of curves. It is useful to observe that the rational cohomology of the orbifold $\sbarm_{g,n}$ coincides with the one of the space $\barm_{g,n}$ (cf. \cite{alr}, Proposition 2.12), and likewise the rational cohomology of $\sm_{g,n}$ is the same as the one of $M_{g,n}$, so that on many occasions we will be able to switch from the orbifold point of view to the space one, and conversely, when needed.

One of the basic results about the homology of $M_{g,n}$
is that it vanishes in high degree. Set
\begin{equation}
c(g,n)=\left\{\aligned n-3\quad &\text{if}\quad g=0\,;\\
4g-5\quad&\text{if}\quad g>0,\,\,\, n=0\,;\\
4g-4+n\quad &\text{if}\quad g>0,\,\,\, n>0\,.
\endaligned\right.
\label{def_c_cell}
\end{equation}
The vanishing theorem, due to John Harer (cf. \cite{harer_2}, Theorem 4.1), is the following:

\begin{theorem}{\label{vanish.high}} $H_k(M_{g,n};\QQ)=0$ for $k>c(g,n)$.
\end{theorem}
\noindent
The case $g=0$ is straightforward.
As we observed in the previous section, $M_{0,n}$ is an affine variety of dimension $n-3$, so that
its homology vanishes in degrees strictly greater than $n-3$.
Actually, a similar proof would yield the general result
if one could establish the following conjecture.

\begin{conjecture}{\label{aff}}{\rm(Looijenga).} Let $g$ and $n$ be non-negative integers such that $2g-2+n>0$. Then $M_{g,n}$ is the union of $g$ affine subsets if $n\neq 0$, and
is the union of $g-1$ affine subsets if $n=0$.
\end{conjecture}
\noindent
We are now going to recall Harer's proof of Theorem \ref{vanish.high}.
From now on we assume $g>1$.
We only treat the case $n=1$. Then we will show how to reduce the other cases to this one.

We fix a compact oriented surface $S$ and a point $p\in S$
and we denote by $\mathcal
A$ the set of isotopy classes, relative to $p$, of loops in $S$ based at $p$. We also require that
no class in $\mathcal A$ represents a homotopically trivial loop in $S$.
The \it arc complex $A=A(S;p)$ \rm is the simplicial complex whose $k$-simplices are given by
$(k+1)$-tuples $a=([\alpha_0],\dots,[\alpha_k])$ of distinct classes in $\mathcal A$ which are representable by a $(k+1)$-tuple
$(\alpha_0,\dots,\alpha_k)$ of loops intersecting only in $p$. The geometric realization of $A$ is denoted by
$|A|$.
A simplex $([\alpha_0],\dots,[\alpha_k])\in A$ is said to be
\it proper \rm if $S\smallsetminus \cup_{i=0}^k\alpha_i$ is a disjoint union of discs.
The improper simplices form a subcomplex of $A$ denoted by $A_\infty$. We set
\begin{equation*}
A^0=A\smallsetminus A_\infty\quad\text{and}\quad|A^0|=|A|\smallsetminus|A_\infty|\,.
\end{equation*}
The mapping class group $\Gamma_{g,1}$ acts on $A$ in the obvious way. For
$a=([\alpha_0],\dots,[\alpha_k])\in A$ and $ [\gamma]\in\Gamma_{g,1}$ we define
\begin{equation}
[\gamma]\cdot a=([\gamma\alpha_0],\dots,[\gamma\alpha_k])\,.\label{cell.action}
\end{equation}
Let us denote by $T_{g,1}$ the Teichm\"uller space of 1-pointed genus $g$ Riemann surfaces. The fundamental result is the following (\cite{be}, Theorem 9.5; see also Chapter 2 of \cite{harer_3} and \cite{ep}).
\begin{theorem}{\label{isopsi}} There is a
$\Gamma_{g,1}$-equivariant homeomorphism
\begin{equation}
\Psi\colon \quad {T}_{g,1}\to |A^0|\,\label{psi_cell}
\end{equation}
\end{theorem}
\noindent We give a very brief sketch of the proof of this theorem. Actually, we will limit ourselves to giving an idea of how
the map (\ref{psi_cell}) is defined. Let $(C, p)$ be a 1-pointed smooth curve of genus $g$.
The uniformization theorem
for Riemann surfaces provides the surface $C\smallsetminus \{p\}$ with a canonical hyperbolic metric, the Poincar\'e metric.
In this metric the point $p$ appears as a \it cusp. \rm This cusp has infinite distance from the points
in $C\smallsetminus\{p\}$.
The set of horocycles around $p$ is a canonical family of simple closed curves in $C\smallsetminus \{p\}$,
contracting to $p$. Choose a small constant
$c$ and a horocycle $\eta$ of length $c$. Given a general point $q\in C\smallsetminus \{p\}$, there is a \it unique \rm shortest geodesic joining $q$ to $\eta$. In fact, the locus of points $q$ in $C$ having the property that there are two or more shortest geodesics from $q$ to $\eta$ is a connected graph $G\subset C$ which is called \it the hyperbolic spine of $(C,p)$. \rm The fundamental property of this graph is that $C\smallsetminus G$ is a disc $\Delta$.
Suppose that $G$ has $k+1$ edges. Then the Poincar\'e dual of
$G$ is a graph consisting of $k+1$ loops $\alpha_0,\dots,\alpha_{k+1}$ based at $p$, which in fact
give rise to a simplex $\alpha=(\alpha_0,\dots,\alpha_{k+1})\in A^0$.
To get a point in $|\alpha|$ we need a $(k+1)$-tuple of numbers $a_1,\dots, a_{k+1}$ with $\sum a_i=1$. Now each loop $\alpha_i$ corresponds, by duality, to an edge $e_i$ of $G$.
Draw the geodesics from the vertices of $e_i$ to $p$. These geodesics
cut out on $\eta$ two arcs. Using elementary hyperbolic geometry one can see that these two arcs have the same length
$a_i$.
This is the coefficient one assigns to the loop $\alpha_i$. We finally choose the constant $c$ in such a way as to get $\sum a_i=1$.

Let us show how Theorem \ref{vanish.high}, for the case $n=1$, follows from Theorem
\ref{isopsi}. We need a general lemma.

\begin{lemma}{\label{simpl.cmplx}}Let $K$ be a simplicial complex.
Let $H\subset K$ be a subcomplex. Set $L=K\smallsetminus H$. Let $K^1$ and $H^1$ be the first barycentric subdivisions of $K$ and $H$, respectively. Let $M$ be the subcomplex of $K^1$ whose vertices are barycenters of simplices
of $L$. In particular, $M$ has dimension $\mu-\nu$, where $\mu$ (resp. $\nu$) is the maximal (resp. minimal) dimension of a simplex
belonging to $L$ and $|M|\subset |L|=|K|\smallsetminus|H|$. Then there is a deformation retraction of $|L|$ onto $|M|$.
Moreover if $K$ is acted on by a group $G$ and if the subcomplex $H$ is preserved by this action,
then the above deformation can be assumed to be $G$-equivariant.
\end{lemma}

\begin{proof}By definition, the vertices of a $k$-simplex of $K^1$ are barycenters $b_0,\dots, b_k$ of a
strictly increasing sequence $a_0<a_1<\dots<a_k$ of simplices of $K$. Such a simplex
is in $M$ if and only if $a_0$ is a simplex of $L$. Let $p\in |L|=|K^1|\smallsetminus|H^1|$.
Write $p=\sum_{i=0}^k\lambda_ib_i$ and let $s$ be the minimum index such that $a_s$
belongs to $L$, so that $s<k$. Setting
\begin{equation*}
p'= \left(\sum_{i=s}^k\lambda_i\right)^{-1}\cdot\sum_{i=s}^k\lambda_i b_i\,, \quad\text{and}\quad F(p,t)=(1-t)p+tp'\,,
\end{equation*}
gives the desired retraction. In the presence of a $G$-action this retraction is clearly $G$-equivariant.
\end{proof}
We now go back to the arc complex $A$.
To apply this lemma to our situation let us estimate the maximum dimension $\mu$ and minimum dimension $\nu$
of simplices belonging to $A^0$. Let then $\alpha=(\alpha_0,\dots,\alpha_k)$ be a simplex in $A^0$ and let $N$ be the number of connected components of $S\smallsetminus\cup_{i=0}^k\alpha_i$.
These components are discs, so that $2-2g=N-k$. In particular $\nu\geq 2g-1$. Suppose now that $\alpha$ has maximal dimension. Then each of the above connected components must be bounded by
3 among the arcs $\alpha_0,\dots,\alpha_k$. It follows that $N=\frac{2}{3}(k+1)$. Therefore
$k=6g-4$ and $\mu-\nu\leq 4g-3$. Using the preceding lemma, we can conclude that
the complex $|A|$ can be $\Gamma_{g,1}$ - equivariantly deformed
to a complex of dimension less than or equal to $4g-3$. This proves
Theorem \ref{vanish.high}, for the case $n=1$

We are now going to treat the cases where $n\neq 1$. Let $(\Sigma; x_1,\dots,x_n)$ be a reference $n$-pointed genus $g$ Riemann surface satisfying the stability condition $2-2g-n<0$. Fix an integer $m\geq 2$.
Consider the Teichm\"uller space $T_{g,n}$, the modular group $\Gamma_{g,n}$,
and its subgroup $\Lambda_m$ formed by isotopy classes of diffeomorphisms of $(\Sigma; x_1,\dots,x_n)$ inducing the identity on $ H_1(\Sigma;\ZZ/(m))$. The quotient $M_{g,n}[m]=T_{g,n}/\Lambda_m$
is, by definition, the moduli space of $n$-pointed, genus $g$ curves with $m$-level structure. Since the homeomorphism $\Psi$ in Theorem
\ref{isopsi} is $\Gamma_{g,n}$-equivariant and therefore, a fortiori,
$\Lambda_m$-equivariant, exactly the same reasoning used to prove
Theorem \ref{vanish.high} in the case $n=1$ shows that $M_{g,1}[m]$ has the homotopy type of a space of dimension at most $4g-3$. This implies that $H_k(M_{g,1}[m];\QQ)$ vanishes for $k>4g-3$ and, more generally, that
\begin{equation}
H_k(M_{g,1}[m];L)=0 \quad\text{for}\quad k>4g-3\,,\label{van_m}
\end{equation}
for any local system of abelian groups $L$. A lemma of Serre \cite{serre} asserts that, when $m\geq 3$, the only
automorphism of an $n$-pointed, genus $g$ curve $(C; p_1,\dots,p_n)$ inducing the identity on $H_1(C;\ZZ/(m))$ is the identity. As a consequence, when $m\geq 3$, $M_{g,n}[m]$ is smooth and equipped with a universal family
$\pi\colon \mathcal C\to M_{g,n}[m]$. From now on we assume that $m\geq 3$.
Since $M_{g,n}$ is the quotient of
$M_{g,n}[m]$ by the finite group $\Gamma_{g,n}[m]=\Gamma_{g,n}/\Lambda_m$, the vanishing of
$H_k(M_{g,n}[m];\QQ)$ implies the vanishing of $H_k(M_{g,n};\QQ)$.
It is then sufficient
to prove our vanishing statement for the homology of $M_{g,n}[m]$.
Look
at the universal family $\pi\colon \mathcal C\to M_{g,n}[m]$. Let $D\subset \mathcal C$ be the divisor
which is the sum of the images of the $n$ marked sections of $\pi$. By definition,
$\mathcal C\smallsetminus D$ is isomorphic to $M_{g,n+1}[m]$. It follows that the morphism
$\eta\colon M_{g,n+1}[m]\to M_{g,n}[m]$ is a topologically locally trivial fibration with fiber $F$ homeomorphic to
$\Sigma\smallsetminus\{x_1,\dots,x_n\}$. Let $L$ be a local system of abelian groups on $M_{g,n+1}[m]$. The $E^2$-term of the Leray spectral sequence with coefficients in $L$ of the fibration $\eta$ is given by
\begin{equation*}
E^2_{p,q}=H_p(M_{g,n}[m]; \mathcal{H}_q(\eta;L))\,,
\end{equation*}
where $\mathcal{H}_q(\eta;L)$ is the local system of the $q$-th homology groups of the fibers of $\eta$ with coefficients in $L$. When $n>0$, the homological dimension of the fiber $F$ is equal to 1. It follows that if
$H_k(M_{g,n}[m];L')$ vanishes for $k>k_0$ and for any local system $L'$, then $H_k(M_{g,n+1}[m];L)$ vanishes for $k>k_0+1$. In view of (\ref{van_m}), this shows, inductively on $n$, that $H_k(M_{g,n}[m];L)$ vanishes for any local system of abelian groups $L$ and any $k>4g-4+n$, whenever $n\geq 1$. In particular, this proves Theorem \ref{vanish.high} for $n\ge 1$.

To treat the case $n=0$ it is convenient, although not strictly necessary, to switch to cohomology. What we have to show, then, is that $H^p(M_{g}[m];\QQ)$ vanishes for $p>4g-5$. Look at the cohomology Leray spectral sequence for $\eta\colon M_{g,1}[m]\to M_{g}[m]$, whose $E_2$ term is
\begin{equation*}
E_2^{p,q}=H^p(M_{g}[m];R^q\eta_*\QQ)\,.
\end{equation*}
In this case, the fibers of $\eta$ are compact Riemann surfaces, and carry a canonical orientation; this gives a canonical section $\omega$, and hence a trivialization, of the rank 1 local system $R^2\eta_*\QQ$. Cupping with $\omega$, wiewed as an element of $E_2^{0,2}=H^0(M_{g}[m];R^2\eta_*\QQ)$, gives homomorphisms $E_2^{p,q}\to E_2^{p,q+2}$, compatible with differentials. In particular, $E_2^{p,0}\to E_2^{p,2}$ is an isomorphism. Thus $d_2\colon E_2^{p,2}\to E_2^{p+2,1}$ vanishes, since $d_2\colon E_2^{p,0}\to E_2^{p+2,-1}$ obviously does, and $d_2\colon E_2^{p-2,3}\to E_2^{p,2}$ also vanishes since $E_2^{p,q}=0$ for $q>2$. The analogues of these statements hold for all $E_r$ with $r\ge 2$; the upshot is that $E_2^{p,2}=E_\infty^{p,2}$. This shows that, if $H^p(M_{g}[m];\QQ)\cong E_2^{p,2}$ does not vanish, neither does $H^{p+2}(M_{g,1}[m];\QQ)$. In conjunction with (\ref{van_m}), this proves our claim.

\medskip
\section{Divisor classes in $\barm_{g,P}$}\label{divclasses}

The first and second rational cohomology groups of $\Mgnbar$ are completely described by the following result (cf. \cite{ac}, Theorem 2.2).

\begin{theorem}{\label{coh_bar}}For any $g$ and $P$ such that
$2g-2+|P|>0$, $H^1(\barm_{g,P};\QQ)=0$ and $H^2(\barm_{g,P};\QQ)$ is generated by the classes $\kappa_1,
\psi_1,\dots,\psi_n$, $\delta_{irr}$, and the classes $\delta_{a,A}$ such that $0\le a\le
g$, $2a-2+\vert A\vert\ge 0$ and $2(g-a)-2+\vert A^c\vert\ge 0$. The relations among
these classes are as follows.
\begin{itemize}
\item[a)]If $g>2$ all relations are generated by those of the form
\begin{equation}
\delta_{a,A}=\delta_{g-a,A^c}\,.\label{complrelbis}
\end{equation}
\item[b)]If $g=2$ all relations are generated by the relations {\rm (\ref{complrelbis})} plus
\begin{equation}
5\kappa_1=5\psi+\delta_{irr}-5\delta_0+7\delta_1\,.\label{kappag2}
\end{equation}
\item[c)]If $g=1$ all relations are generated by the relations {\rm (\ref{complrelbis})} plus the following
ones
\begin{eqnarray}
\kappa_1&=&\psi-\delta_0\,,\label{kappag1} \\
12\psi_p&=&\delta_{irr}+12\sum_{{S\ni p}\atop{\vert S\vert\ge
2}}\delta_{0,S}\,.\label{psig1}
\end{eqnarray}
\item[d)]If $g=0$, all relations are generated by {\rm (\ref{complrelbis})}, by the relations
\begin{equation}
\sum_{A\ni p,q\atop A\not\ni r,s}\delta_{0,A}
=\sum_{A\ni p,r\atop A\not\ni q,s}\delta_{0,A}
=\sum_{A\ni p,s\atop A\not\ni q,r}\delta_{0,A}\,,\label{Keelrel}
\end{equation}
where $\{p,q,r,s\}$ runs over all quadruples of distinct elements in $P$, and by
\begin{equation}
\aligned
\kappa_1&=\sum_{A\not\ni x,y}(\vert A\vert-1)\delta_{0,A}\,\\
\psi_z&=\sum_{A\ni z\atop A\not\ni x,y}\delta_{0,A}\,,\\
\delta_{irr}&=0\,.
\endaligned
\label{kpsidirr_rel}
\end{equation}
\end{itemize}
\noindent Furthermore, the rational Picard group $\operatorname{Pic}(\barm_{g,P})\otimes \QQ$ is always isomorphic to $H^2(\barm_{g,P};\QQ)$.
\end{theorem}

\noindent
We will give a brief sketch of the proof of Theorem \ref{coh_bar}, and more precisely of the fact that
tautological classes generate $H^2(\Mgnbar;\QQ)$. The case $g=0$ is of course part of Keel's theorem. Let us then assume that $g>0$.
By Poincar\'e duality, which holds with rational coefficients since $\sm_{g,n}$ is a smooth orbifold, we can express Theorem \ref {vanish.high} in terms of cohomology with compact support. We set
\begin{equation*}
d(g,n)=\left\{\aligned n-4\quad&\text{if}\quad g=0\,;\\
2g-2\quad&\text{if}\quad g>0,\,\,\, n=0\,;\\
2g-3+n\quad&\text{if}\quad g>0,\,\,\, n>0\,.
\endaligned\right.
\end{equation*}
Dualizing Theorem \ref {vanish.high} we then get:
\begin{equation}
H^k_c(M_{g,n};\QQ)=0\,, \quad\text{for}\quad k\leq d(g,n)\label{vanish.low}
\end{equation}
We look at the exact sequence of \it rational \rm cohomology with compact support
\begin{equation*}
\dots\to H^k_c(\Mgn)\to H^k(\Mgnbar )\to H^k(\partial\Mgn)\to H^{k+1}_c(\Mgn)\to\dots
\end{equation*}
Here and in the sequel, when we omit mention of the coefficients in cohomology, we always assume $\QQ$-coefficients.
Using (\ref {vanish.low}), it follows that

\begin{proposition}{\label{inj.cell}} The restriction map
\begin{equation*}
H^k(\Mgnbar )\to H^k(\partial\Mgn)
\end{equation*}
is injective for $ k\leq d(g,n)$, and is an isomorphism for $ k<d(g,n)$.
\end{proposition}
\noindent
Let us now consider the irreducible components of the boundary $\partial\Mgn$. As we observed in section \ref{strata},
each of these components is the image of a map
$\mu_i\colon X_i\to\Mgnbar$ where $X_i$ can be of two different kinds. Either $X_i=\overline{M}_{g-1,n+2}$,
and $\mu_i$ is obtained by identifying the last two marked points of each $(n+2)$-pointed curve of genus $g-1$,
or $X_i=\overline{M}_{p,a+1}\times \overline{M}_{q,b+1}$, where $p+q=g$, $a+b=n$, $2p-2+a\geq 0$,
$2q-2+b\geq 0$ and
$\mu_i$ is obtained by identifying the $(a+1)$-st marked point of an $(a+1)$-pointed genus $p$ curve with the
$(b+1)$-st marked point of a $(b+1)$-pointed genus $q$ curve.
Let us then consider the composition
\begin{equation*}
\sigma\colon \coprod_{i=1}^N X_i\to\partial\Mgn\to\Mgnbar\,.
\end{equation*}
One can prove a stronger version of Proposition \ref{inj.cell} stating that, in degrees not bigger than
$d(g,n)$, the morphism $\sigma$ also induces an injection in cohomology. This result is somewhat surprising, especially when
the degree $k$ is less than $d(g,n)$. In this case, it implies that the degree $k$ cohomology of $\partial\Mgn$ injects in
the degree $k$ cohomology of the disjoint union of the $X_i$. This shows that the geometry arising from the way in which the various strata of $\partial\Mgn$ intersect each other does not contribute to the cohomology, at least in relatively low range.
\vskip 0.3 cm
\begin{theorem}{\label{inj.2.cell}}The composition map $\sigma$ induces an injective homomorphism
\begin{equation}
\gamma\colon H^k(\Mgnbar)\to \overset{N}{\underset{i=1}\oplus}H^k(X_i)
\label{gamma}
\end{equation}
for $ k\leq d(g,n)$.
\end{theorem}

\begin{proof}As proved in \cite{loo}, Section 2, and \cite{boggi}, Proposition 2.6, the moduli spaces $\Mgnbar$ are quotients of smooth complete varieties modulo the action of finite groups (see also \cite{acv} or \cite{gacII}, Chapter 17). Thus the same is true for each of the $X_i$; we write $X_i=Z_i/G_i$, where $Z_i$ is smooth and $G_i$ is a finite group. It will then suffice to prove the injectivity of the map
\begin{equation*}
H^k(\Mgnbar)\to \overset{N}{\underset{i=1}\oplus}H^k(Z_i)
\end{equation*}
in the given range. For this we use the following result in Hodge theory, due to Deligne (\cite{deligne3}, Proposition (8.2.5)).

\begin{theorem}{\label{deligne}} Let $Y$ be a complete variety. If $X\to Y$ is a proper surjective morphism and $X$ is smooth,
then the weight $k$ quotient of $H^k(Y;\QQ)$ is the image of $H^k(Y;\QQ)$
in $H^k(X;\QQ)$.
\end{theorem}
\noindent
We use this result by taking as $Y$ the boundary $\partial\Mgn$ and as $X$ the disjoint union of the $Z_i$.
In particular, $H^k(\partial\Mgn)/W_{k-1}(H^k(\partial\Mgn))$ injects in $H^k(X)$.
By the previous proposition we know that, in the given range, the map
\begin{equation*}
\rho\colon H^k(\Mgnbar)\to H^k(\partial\Mgn)
\end{equation*}
is injective. We then have to show that the intersection of $W_{k-1}(H^k(\partial\Mgn))$ with the image of $\rho$ is zero.
But this is evident because $\rho$ is a morphism of mixed Hodge structures, and hence strictly compatible with filtrations.
In fact
\begin{equation*}
\rho(H^k(\Mgnbar))\cap W_{k-1}(H^k(\partial\Mgn))=\rho (W_{k-1}(H^k(\Mgnbar)))=\rho(\{0\})=0\,,
\end{equation*}
since $H^k(\Mgnbar)$ is of pure weight $k$. The proof of Theorem \ref{inj.2.cell} is now complete.
\end{proof}

An elementary application of Theorem \ref{inj.2.cell} is the following.

\begin{corollary}{\label{h1.cell}} For all $g$ and $n$ such that $2g-2+n>0$, $H^1(\Mgnbar;\QQ)=0$.
\end{corollary}
\begin{proof}The theorem is true for $\barm_{0,3}=\{pt\}$ and $\barm_{1,1}=\PP^1$.
Except in these two cases, $1\geq d(g,n)$, so that the homomorphism
$H^1(\Mgnbar)\to \overset{N}{\underset{i=1}\oplus}H^1(X_i)$ is always injective. Now $X_i$ is either
the image of $\overline{M}_{g-1,n+2}$
or the image of $\overline{M}_{p,a+1}\times \overline{M}_{q,b+1}$, where $p+q=g$, $a+b=n$.
Therefore, by the K\"unneth formula, $H^1(\Mgnbar)$ injects in a direct sum of first cohomology groups of moduli spaces
$\overline{M}_{p,\nu}$ where either $p=g-1$ or $p=g$ and $\nu<n$. The result follows by double induction on $g$ and $n$.
\end{proof}

The proof of the statement concerning $H^2(\Mgnbar)$ in Theorem \ref{coh_bar} also proceeds by double induction on $(g,n)$, starting from the initial cases $(1,1)$ and $(1,2)$, which can be easily treated directly.

The strategy for the inductive step is quite simple and boils down to elementary linear algebra. The idea is to use Theorem \ref{inj.2.cell}.
Suppose we want to show that
$H^2(\Mgnbar)$ is generated by tautological classes, assuming the same is known to be
true in genus less then $g$, or in genus $g$ but with fewer than $n$ marked points.
Look at the injective homomorphism (\ref{gamma}) and denote by $\gamma_i$ the composition of $\gamma$ with the projection onto $H^2(X_i)$.
By induction hypothesis, each summand $H^2(X_i)$ is
generated by tautological classes, all relations among which are known. Since one has a complete control on the effect of
each map $H^2(\Mgnbar)\to H^2(X_i)$ on tautological classes,
the subspace of $H^2(X_i)$, generated by the images of the tautological classes in $H^2(\Mgnbar)$
is known.
On the other hand, given any class in
$H^2(\Mgnbar)$, its restrictions to the $X_i$
satisfy obvious compatibility relations on the ``intersections'' of the $X_i$. The subspace $W\subset\bigoplus H^2(X_i)$ defined by these compatibility relations can be
completely described because the spaces $H^2(X_i)$ are generated by tautological classes. By elementary linear algebra computations one shows that
$W$ is equal to the image of $H^2(\Mgnbar)$ in
$\bigoplus H^2(X_i)$. Since $H^2(\Mgnbar)$ is known to inject in
$\bigoplus H^2(X_i)$, this concludes the proof that $H^2(\barm_{g,P};\QQ)$ is as described in the statement of Theorem \ref{coh_bar}.

As for the last statement of the theorem, notice first that, as is the case for rational cohomology, the rational Picard group of the space $\barm_{g,n}$ coincides with the one of the smooth orbifold $\sbarm_{g,n}$. But then $\operatorname{Pic}(\sbarm_{g,n})\otimes\QQ\to H^2(\sbarm_{g,n};\QQ)$ is onto since $H^2(\sbarm_{g,n};\QQ)$ is generated by divisor classes, and is injective since $H^1(\sbarm_{g,n};\QQ)$ always vanishes.

\bigskip
We finally recall that the actual Picard group of $\sbarm_{g,n}$ is also known. The following result is proved in \cite{acpic} (Theorem 2, page 163), using the results of \cite{harer}.

\begin{theorem}{\label{picz}}For all $g\ge 3$ and all $n$, $\operatorname{Pic}(\sbarm_{g,n})$ is freely generated by $\lambda$, the $\psi_i$, and the boundary classes.
\end{theorem}

\noindent It is possible to explicitly describe $\operatorname{Pic}(\sbarm_{g,n})$ for $g<3$, but we will not do it here; we just observe that the case $g=0$ is covered by Keel's theorem in Section \ref{genus0}.

\medskip
\section{Deligne's spectral sequence}\label{deligneseq}

We recall Deligne's ``Gysin'' spectral sequence
computing the cohomology of non-singular varieties (see \cite{deligne2}, 3.2.4.1).
Let $X$ be a complex manifold and let $D$ be a divisor with normal crossings in $X$. Set $V=X\smallsetminus D$.
The Gysin spectral sequence we are about to define will compute the cohomology of $V$.
Locally inside $X$, the divisor $D$ looks like a union of coordinate hyperplanes. Denote by
$D^{[p]}$ the union of the points of multiplicity at least $p$ in $D$ and by $\widetilde D^{[p]}$
the normalization of $D^{[p]}$.
We also set
\begin{equation*}
\widetilde D^{[0]}=D^{[0]}=X\,.
\end{equation*}
Concretely, a point $y$ of $\td^{[p]}$ can be thought of as the datum of a point $x$ in $D^{[p]}$ and of $p$ local components of $D$ through $x$. We denote by $E_p(y)$ the set of these components.
The sets $E_p(y)$ form a local system
on $\widetilde D^{[p]}$ which we denote by $E_p$.
The set of orientations of $E_p(y)$ is by definition the set of generators of $\wedge^p\ZZ^{E_p(y)}$. The local system of orientations of $E_p$
defines on $\widetilde D^{[p]}$ a local system of rank one
\begin{equation*}
\varepsilon^p=\wedge^p\QQ^{E_p}
\end{equation*}
While $D^{[p]}$ is clearly included in $D^{[p-1]}$, in general there is no natural map from $\td^{[p]}$ to $\td^{[p-1]}$. What we have instead is a natural correspondence between the two. We let $\td^{[p-1,p]}$ be the space whose points are pairs $(y,L)$, where $y=(x,E_p(y))\in \td^{[p]}$ and $L\in E_p(y)$. We then have morphisms
\begin{equation*}
\xymatrix
{
\td^{[p-1,p]}\ar[d]^{\pi^{[p-1,p]}}\ar[rr]^{\xi^{[p-1,p]}}&&\td^{[p-1]}\\
\td^{[p]}&
}
\end{equation*}
\begin{equation*}
\pi^{[p-1,p]}(y,L)=y\,;\qquad \xi^{[p-1,p]}(y,L)=(x,E_p(y)\smallsetminus\{L\})\,.
\end{equation*}
Moreover, there is a natural isomorphism
\begin{equation*}
({\pi^{[p-1,p]}})^*\varepsilon^p \cong ({\xi^{[p-1,p]}})^*\varepsilon^{p-1}
\end{equation*}
This makes it possible to define a generalized Gysin homomorphism
\begin{equation*}
\xymatrix{
H^{i}(\widetilde D^{[p]};\varepsilon^p)\ar[r]^{{\pi}^*\hskip70pt}&
H^{i}(\widetilde D^{[p-1,p]};{\pi}^*\varepsilon^p)\cong
H^{i}(\widetilde D^{[p-1,p]};{\xi}^*\varepsilon^{p-1})\ar[r]^{\hskip55pt{\xi}_*}&
H^{i+2}(\widetilde D^{[p-1]};\varepsilon^{p-1})
}
\end{equation*}
where, for brevity, we have written $\pi$ and $\xi$ for $\pi^{[p-1,p]}$ and $\xi^{[p-1,p]}$, respectively.

\bigskip
The following theorem by Deligne holds.
\begin{theorem}{\label{del_spec}}There is a spectral sequence, abutting to $H^*(V;\QQ)$ and with
$E_2=E_\infty$,
whose $E_1$-term is given by
\begin{equation*}
E_1^{-p,q}=\left\{\aligned
&H^{q-2p}(\widetilde D^{[p]};\varepsilon^p)\quad\text{for}\quad p>0\,,\\
&H^q(X;\QQ)\qquad\,\,\quad\text{for}\quad p=0\,,\\
&0\qquad\qquad\,\qquad\quad\text{for}\quad p<0\,.
\endaligned\right.
\end{equation*}
Moreover, the differential
\begin{equation*}
d_1\colon \,H^{q-2p}(\widetilde D^{[p]};\varepsilon^p)\to
H^{q-2p+2}(\widetilde D^{[p-1]};\varepsilon^{p-1})
\end{equation*}
is the Gysin homomorphism $({\xi^{[p-1,p]}})_*({\pi^{[p-1,p]}})^*$.
\end{theorem}
\noindent
The above theorem also holds in the case in which $X$ is an orbifold and $D$ an orbifold divisors with normal crossing.
Indeed, in view of the local nature of their proofs, one can immediately see that Proposition 3.6, Proposition 3.13 in \cite{delignereg}
and Proposition 3.1.8 in
\cite{deligne2} all
hold in the orbifold situation.
We can then apply the preceding theorem to the situation in which
$X=\sbarm_{g,P}$, $D=\partial \sm_{g,P}$ and $V=\sm_{g,P}$. To state the result one obtains we choose, once and for all, a representative in each isomorphism class of $P$-pointed, genus $g$ stable graphs with exactly $p$ edges, and let $\mathcal G^p_{g,P}$ denote the (finite) set of these representatives. We then get the following.

\begin{theorem}{\label{deligne_mgnbar}} There is a spectral sequence, abutting to $H^*(M_{g,P};\QQ)$
and with
$E_2=E_\infty$,
whose $E_1$-term is given by
\begin{equation*}
E_1^{-p,q}=\left\{\aligned
&\bigoplus_{\Gamma\in \mathcal G^p_{g,P}} H^{q-2p}(\widetilde D_{\Gamma}; \varepsilon ^p)\quad\text{for}\quad p>0\,,\\
&H^q(\barm_{g,P};\QQ)\qquad\,\,\qquad\text{for}\quad p=0\,,\\
&0\qquad\qquad\qquad\qquad\ \quad\text{for}\quad p<0\,.
\endaligned
\right.
\end{equation*}
Moreover, with the notation introduced at the end of section \ref{strata}, the differential
\begin{equation*}
d_1\colon \bigoplus_{\Gamma\in \mathcal G^p_{g,P}} H^{q-2p}(\widetilde D_{\Gamma};\varepsilon^p)\to
\bigoplus_{\Gamma'\in \mathcal G^{p-1}_{g,P}} H^{q-2p+2}(\widetilde D_{\Gamma'}; \varepsilon^{p-1})
\end{equation*}
of this spectral sequence is the Gysin map
\begin{equation*}
\sum_{\Gamma\in \mathcal G^p_{g,P}\,,\,\Gamma'\in \mathcal G^{p-1}_{g,P}\,,\, \Gamma<\Gamma'} (\widetilde\xi_{\Gamma',\Gamma})_*(\widetilde\pi_{\Gamma',\Gamma})^*
\end{equation*}
\end{theorem}

\medskip
\section{Computing $H^1(M_{g,n})$ and $H^2(M_{g,n})$}\label{H1H2}

As an application of Deligne's spectral sequence, we are going to deduce from Theorem \ref{coh_bar} two classical results on $H^1(M_{g,n};\QQ)$ and
$H^2(M_{g,n};\QQ)$. The first of these is due to Mumford (\cite {mum}, Theorem 1) for $n=0$ and to Harer (\cite{harer}, Lemma 1.1) for arbitrary $n$, while the second is due to Harer \cite{harer}.

\begin{theorem}\label{mupoha}The following hold true:
\begin{enumerate}
\item (Mumford, Harer) $ H^1(M_{g,n};\QQ)=0$ for any $g\ge 1$ and any $n$ such that $2g-2+n>0$.
\item (Harer) $ H^2(M_{g,n};\QQ)$ is freely generated by $\kappa_1,\psi_1,\dots,\psi_n$ for any $g\ge 3$ and any $n$. $H^2(M_{2,n};\QQ)$ is freely generated by $\psi_1,\dots,\psi_n$ for any $n$, while $H^2(M_{1,n};\QQ)$ vanishes for all $n$.
\end{enumerate}
\end{theorem}
\noindent Before giving the proof, let us remark that both parts of the theorem can be considerably strengthened. Teichm\"uller's theorem implies that the orbifold fundamental group of $\sm_{g,n}$ is the Teichm\"uller modular group $\Gamma_{g,n}$. On the other hand, it is known that $\Gamma_{g,n}$ equals its commutator subgroup for $g>2$. For $n=0$ this is Theorem 1 of \cite{pow}, while for arbitrary $n$ it is Lemma 1.1 of \cite{harer}. Thus the first integral homology group of $\sm_{g,n}$ vanishes for $g>2$. Likewise, the main result of \cite{harer} actually computes the second integral homology of $\sm_{g,n}$ for $g> 4$; this turns out to be free of rank $n+1$ for any $n$.

\begin{proof}[Proof of Theorem \ref{mupoha}]It is convenient to adopt the $M_{g,P}$ notation instead of the $M_{g,n}$ one. In the course of the proof we shall omit mention of the coefficients in cohomology, assuming rational coefficients throughout. We let $D_1,\dots,D_N$ be the components of the boundary of $\sbarm_{g,P}$, so that, in the notation of the previous section, $\cup D_i=D^{[1]}$. Clearly, $\td^{[1]}=\coprod \td_i$, where $\td_i$ stands for the normalization of $D_i$. We first compute $H^1(M_{g,P})$.
The only non-zero terms in Deligne's spectral sequence that are relevant to the computation of
$H^1(M_{g,P})$ are $E_1^{-1,2}$, $E_1^{0,2}$, $E_1^{0,1}$.
On the other hand, $E_1^{0,1}$ equals $H^1(\barm_{g,P})$ which, by
Theorem \ref{coh_bar}, vanishes. It follows that
\begin{equation*}
H^1(M_{g,P})=\ker\{d_1^{-1,2}\colon E_1^{-1,2}\to E_1^{0,2} \}
\end{equation*}
But $d_1^{-1,2}$ is the Gysin map
\begin{equation}
\bigoplus_{i=1}^{N} H^0(\widetilde D_i)\to H^2(\barm_{g,P})\label{Gysin}
\end{equation}
From Theorem \ref{coh_bar} we know that the boundary classes in $H^2(\barm_{g,P})$
are linearly independent as long as $g>0$. This means that $d_1^{-1,2}$ is injective. This proves the vanishing of $H^1(M_{g,P})$ when $g>0$.
We next consider the second rational cohomology group of $M_{g,P}$. The only non-zero terms in Deligne's spectral sequence that are relevant to the computation of
$H^2(M_{g,P})$ are $E_1^{-2,4}$, $E_1^{-1,4}$, $E_1^{-1,3}$ and $E_1^{0,2}$.
Since
\begin{equation*}
E_1^{-1,3}=\underset{i=1}{\overset{N} \oplus}H^1(\widetilde D_i)\,,
\end{equation*}
and since each $\widetilde D_i$ is the quotient of a product of moduli spaces of stable pointed curves by a finite group, the term
$E_1^{-1,3}$ vanishes by Theorem \ref{coh_bar}. We then have
\begin{equation*}
H^2(M_{g,P})=\ker (d_1^{-2,4})\oplus \operatorname{coker} (d_1^{-1,2})\,.
\end{equation*}
From Theorem \ref{coh_bar} it follows that
\begin{equation*}
\operatorname{coker}(d_1^{-1,2})=\left\{
\begin{matrix}
{\QQ}\langle \kappa_1,\psi_1,\dots,\psi_n\rangle \hfill&&\text{for } g>2,\hfill\\
{\QQ}\langle \psi_1,\dots,\psi_n\rangle \hfill&&\text{for } g=2,\hfill\\
0\hfill&&\text{for } g=1.\hfill
\end{matrix}
\right.
\end{equation*}
It remains to show that $\ker (d_1^{-2,4})=0$. The homomorphism
\begin{equation*}
d_1^{-2,4}\colon H^0(\widetilde D^{[2]};\varepsilon^2)\to H^2(\widetilde D^{[1]};\varepsilon^1)
\end{equation*}
is the Gysin map
\begin{equation}
\bigoplus_{\Gamma\in \mathcal{G}_{g,P}^{2}} H^0(\widetilde D_\Gamma;\varepsilon^2)
\to \bigoplus_{i=1}^{N} H^2(\widetilde D_i)\,.\label{Gysin_2}
\end{equation}
The graphs of $P$-pointed, genus $g$ stable curves with two double points
are all illustrated in Figure \ref{graphs}, where $C\cup D\cup B=A\cup B=P$ and $c+d+b=a+b=g$.
\begin{figure}
\includegraphics[height=130pt]{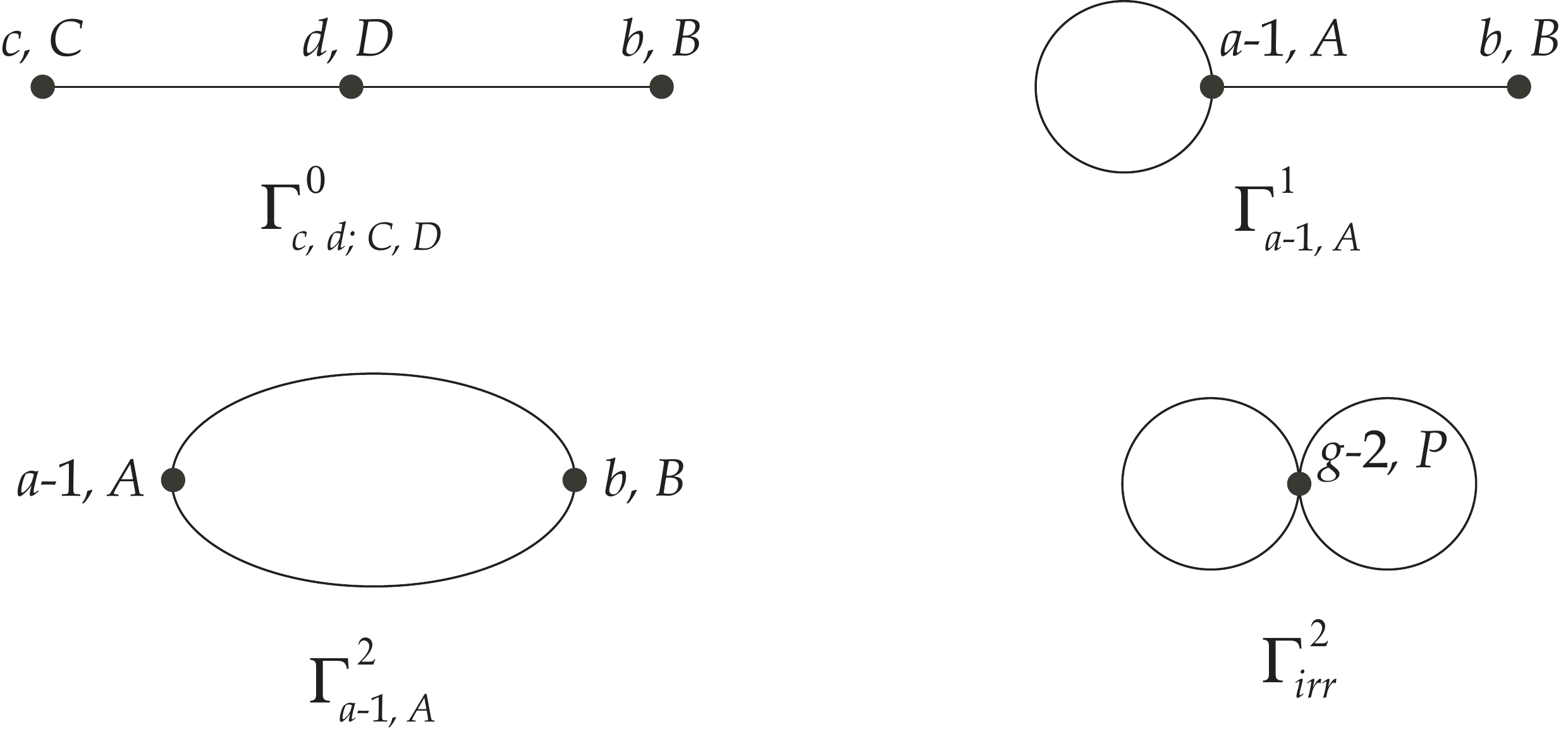}
\caption{}
\label{graphs}
\end{figure}
We now consider an element $\delta\in H^0(\widetilde D^{[2]};\varepsilon^2)$.
Notice that, when $D_\Gamma$ is part of the self-intersection of one of the $D_i$, we have $H^0(\widetilde D_\Gamma;\varepsilon^2)=0$ since $\varepsilon^2$ is not trivial. This rules out the graphs $\Gamma^0_{c,g-2c;\emptyset,P}, \Gamma^2_{a-1,A}, \Gamma^2_{irr}$. In all other cases
$H^0(\widetilde D_\Gamma;\varepsilon^2)=H^0(\widetilde D_\Gamma)$. We may then write
\begin{equation*}
\delta=\sum \alpha^0_{c,d; C,D}\,1_{\Gamma^0_{c,d; C,D}}+\sum \alpha^1_{a-1, A}\,1_{\Gamma^1_{a-1, A}}\,,
\end{equation*}
where $1_\Gamma$ is a generator of $H^0(\widetilde D_\Gamma)$, and where
$(d, D)\neq (g-2c,P)$.
Proceeding by contradiction, we assume that the image of $\delta$ under (\ref{Gysin_2}) is zero. Let $g=a+b$ and
assume that $a>0$. Fix a subset $A\subset P$. Set
$B=P\smallsetminus A$, and assume $(a,A)\neq(g,P)$. By K\"unneth's formula, one of the summands $H^2(\widetilde D_i)$ in the right-hand side of (\ref{Gysin_2}) contains a summand of the form
$H^2(\barm_{a,A\cup\{x\}})\otimes H^0(\barm_{b,B\cup\{y\}})$.
Let $\pi_{a,A}$ be the composition of the map (\ref{Gysin_2}) with the projection onto this summand.
This homomorphism vanishes identically on a certain number of summands in the left-hand side of (\ref{Gysin_2}).
Taking this into account,
$\pi_{a,A}$ can be viewed as a homomorphism
\begin{equation*}
\pi_{a,A}\colon \left(\bigoplus_{c+d=a,\,\, C\cup D=A } H^0\left(\widetilde D_{\Gamma^0_{c,d; C,D}}\right)\right)\oplus
H^0\left(\widetilde D_{\Gamma^1_{a-1; A}}\right)\to H^2(\barm_{a,A\cup\{x\}})\,.
\end{equation*}
The summands in the domain of $\pi_{a, A}$
play, with respect to $\barm_{a,A\cup\{x\}}$, the same role that the summands $H^0(\widetilde D_1),\dots, H^0(\widetilde D_N)$
play for $\barm_{g,P}$ in the Gysin homomorphism (\ref{Gysin}). Since $a>0$, the homomorphism $\pi_{a, A}$
is injective for all pairs $(a,A)$ with $a>0$ and $(a,A)\neq(g,P)$. It follows that the coefficients $\alpha^0_{c,d; C,D}$,
$\alpha^1_{a-1; A}$ are all zero.
\end{proof}

\bibliographystyle{amsplain}

\providecommand{\bysame}{\leavevmode\hbox to3em{\hrulefill}\thinspace}

\end{document}